\documentclass{elsarticle}

\usepackage{amssymb,amsfonts,verbatim,graphicx,mathtools,mathrsfs,commath,float,tikz-cd,amsthm,setspace}

\usepackage{tikz}
\usepackage{enumitem}
\usepackage[T1]{fontenc}
\usepackage{wrapfig}
\usepackage[hidelinks]{hyperref}
\usepackage{xfrac}

\usepackage[margin=1in]{geometry}

\newtheorem{ut}{Theorem}
\newtheorem{ul}[ut]{Lemma}
\newtheorem{uc}[ut]{Corollary}

\newtheorem*{uq}{Question}
\newtheorem{ue}{Example}

\makeatletter
\def\ps@pprintTitle{%
 \let\@oddhead\@empty
 \let\@evenhead\@empty
 \def\@oddfoot{\centerline{\thepage}}%
 \let\@evenfoot\@oddfoot}
\makeatother

\theoremstyle{remark}
\newtheorem*{ur}{Remark}
\newtheorem*{urs}{Remarks}

\makeatletter
\g@addto@macro\th@definition{\thm@headpunct{\textnormal{.}}}
\makeatother
\theoremstyle{definition}
\newtheorem*{ucl1}{\textnormal{\textit{Claim 1}}}
\newtheorem*{ucl2}{\textnormal{\textit{Claim 2}}}
\newtheorem*{ucl3}{\textnormal{\textit{Claim 3}}}
\newtheorem*{ucl4}{\textnormal{\textit{Claim 4}}}
\newtheorem*{ucl5}{\textnormal{\textit{Claim 5}}}
\newtheorem*{uca1}{\textnormal{\textit{Case 1}}}
\newtheorem*{uca2}{\textnormal{\textit{Case 2}}}

\newtheorem{innercustomthm}{Question}
\newenvironment{uconj}[1]
  {\renewcommand\theinnercustomthm{#1}\innercustomthm}
  {\endinnercustomthm}

\DeclareMathOperator{\cl}{cl}

\def\bysame{\leavevmode\hbox to3em{\hrulefill}\thinspace}

\numberwithin{equation}{section}

\date{\today}

\begin{document}

\begin{frontmatter}

\title{On indecomposability  of $\beta X$}

\author{David Sumner Lipham}
\ead{dsl0003@auburn.edu}
\address{Department of Mathematics, Auburn University, Auburn, AL 36830}

\begin{abstract}The following is an open problem in topology: Determine whether the Stone-\v{C}ech compactification of a widely-connected space is necessarily an indecomposable continuum. Herein we describe  properties of $X$ that are necessary and sufficient in order for $\beta X$ to be  indecomposable. We show that  indecomposability and irreducibility are equivalent properties in compactifications of  widely-connected separable metric spaces, leading to some equivalent formulations of the  open problem.  We also construct a widely-connected subset of Euclidean $3$-space which is contained in a composant of each of its compactifications.  The example answers a question of Jerzy Mioduszewski.
\end{abstract}


\end{frontmatter}

\section{Introduction}

This paper addresses a collection of problems relating to widely-connected sets and indecomposability of the Stone-\v{C}ech compactification.\footnote{Indecomposability of the growth $\beta X\setminus X$ (otherwise known as the Stone-\v{C}ech remainder) was characterized  in \cite{dic}, closely following the discovery that $\beta [0,\infty)\setminus [0,\infty)$ is an indecomposable continuum -- see \cite{belll} and \cite{KH} \S9.12. Here we are interested in $\beta X$ as a whole.}  

A connected topological space $W$ is \textit{widely-connected} if every non-degenerate connected subset of $W$ is dense in $W$.  A connected compact Hausdorff  space  (a \textit{continuum}) is \textit{indecomposable} if it cannot be written as the union of any two of its proper subcontinua. By Theorem 2 in \cite{kur} \S 48 V, the latter term can be consistently defined in the absence of compactness and/or connectedness. To wit, a topological space $X$ is \textit{indecomposable} if every connected subset of $X$ is either dense or nowhere dense in $X$ (cf. \cite{swi2,swi3,rud3} for connected spaces).  Now every widely-connected space is indecomposable. 
  
The three-part question below was asked by  Jerzy Mioduszewski at the 2004 Spring Topology and Dynamics Conference. 

\begin{uq}[Mioduszewski; 23 in \cite{rep}]\label{mio}Let $W$ be a widely-connected space. 
\begin{enumerate}[label=\textnormal{(\Alph*)}]	
\item  Is $\beta W$ necessarily an indecomposable continuum? 
\item  If $W$ is  metrizable and separable, does $W$ necessarily have a metric compactification
	which is an indecomposable continuum?
\item If $W$ is metrizable and separable, does $W$ necessarily have a metric compactification
$\gamma W$ such that for every composant $P$ of $\gamma W$, $W\cap P$ is \textnormal{(i)} hereditarily
disconnected? \textnormal{(ii)}  finite? \textnormal{(iii)} a singleton?
\end{enumerate}\end{uq} 

\noindent Part (A) later became Problem 521 in  \textit{Open Problems in Topology II}, due to David Bellamy \cite{bel}.  In Problem 520 from the same book, Bellamy conjectured a positive answer to  (B).   

Note that, as stated, (A) is more general than (B) and (C). Question (A) is about arbitrary Tychonoff spaces  and Questions (B) and (C) are about separable metrizable spaces. We let (A$'$) be the version of Question (A) that assumes $W$ is separable  and metrizable.

 \subsection{Notation and terminology}

 In Mioduszewski's question: 
 \begin{itemize}
\item $\beta W$ denotes the Stone-\v{C}ech compactification of $W$;
\item $\gamma W$  is a \textit{compactification} of $W$ if  $\gamma W$ is a compact Hausdorff space in which $W$ is densely embedded;
\item $P$ is a \textit{composant} of  $\gamma W$ if $P$ is the union of all proper subcontinua of $\gamma W$ that contain a given point; 
\item $W\cap P$ is \textit{hereditarily disconnected} means that  $\abs{C}\leq 1$ for every connected $C\subseteq W\cap P$. 
\end{itemize}
These definitions generalize in the obvious ways; see \cite{eng}.  See \cite{KH} for constructions and unique properties of the Stone-\v{C}ech compactification. Basic information about  composants is given in  \cite{kur} \S48 VI.

A subset $Q$  of a topological space $X$ is called a \textit{quasi-component} of $X$ if there exists $q\in X$ such that $$Q=\bigcap \{A\subseteq X:A\text{ is clopen and }q\in A\}.$$ If $|Q|=1$ for every quasi-component $Q$ of $X$, then $X$ is \textit{totally disconnected}.  

If $p$ and $q$ are two points in a connected space $X$,  then $X$ is \textit{reducible between $p$ and $q$} if there is a closed connected $C\subsetneq X$ with $\{p,q\}\subseteq C$. Otherwise, $X$ is \textit{irreducible between $p$ and $q$}.  Observe that $W$ is widely-connected if and only if $W$ is connected and irreducible between every two of its points.   

A continuum with only one composant is said to be  \textit{reducible}.  A continuum is \textit{irreducible} if it has more than one composant, that is, if there are two points between which the continuum is irreducible. 

\subsection{Summary of results and main example}

Our results are  divided across two sections. Results in Section \ref{people} apply to general Tychonoff spaces, while Section \ref{lop} is reserved for the separable metrizable setting.  

For Tychonoff $X$, we  characterize indecomposability of $\beta X$ via an elementary property of $X$ (Theorem \ref{4} \& Corollary \ref{5}).  We also prove $\beta X$ is indecomposable [resp. irreducible] if $X$ has an indecomposable [resp. irreducible] compactification (Theorems \ref{6}(i) \& \ref{6}(ii)). And irreducible compactifications of indecomposable spaces are  indecomposable (Theorem \ref{6}(iii)).   

Conversely, indecomposable connected compactifications of \textit{separable metrizable}  spaces are  irreducible  (Theorem \ref{8}).   And if $X$  is connected,  separable, and metrizable, and $\beta X$ is indecomposable, then $X$ densely embeds into an  indecomposable subcontinuum of the Hilbert cube (Theorems \ref{4} \&  \ref{9}).

The preceding theorems imply (A$'$) and (B) are equivalent to:

\begin{uconj}{\textnormal{(D)}}If $W$ is a connected separable metric space that is irreducible between every two of its points, then does $W$ necessarily have an irreducible compactification?\footnote{There does exist an irreducible connected plane set every compactification of which is reducible -- see the end of our short follow-up paper \cite{lip2}.  The example is not indecomposable, although it contains an indecomposable connected set. }\end{uconj}

\begin{figure}[H]
 \begin{center}
\begin{tikzpicture}[baseline= (A).base]
\node[scale=1] (A) at (0,0){
\begin{tikzcd}
    \text{(A$'$)}  \rar[Rightarrow,"\hspace{-0.5mm}\text{Thms. \ref{4} \& \ref{9}}\hspace{0.5mm}"{yshift=2pt}]  \arrow[ddr,Leftarrow,"\text{\hspace{-.7cm}\rotatebox{-34.3}{Thm. \ref{6}}}"{yshift=-6.5pt,xshift=10.5pt}]
    & [3em] \text{(B)}  \arrow[dd,Rightarrow,"\text{\rotatebox{-90}{Thm. \ref{8}}}"{yshift=2pt,xshift=2pt}] 
    & [2em]  \text{(C)}
    \arrow[ddl,Rightarrow]& \hspace{-3em}\footnotesize\text{(i, ii, or iii)}
    \\
    &   \\&\text{(D)}
    \end{tikzcd}
};
\end{tikzpicture}
\end{center}
\caption{Summary (Corollary \ref{10}); \ `(X)$\implies$(Y)' means `a positive answer to (X) implies a positive answer to (Y)'}
\end{figure}
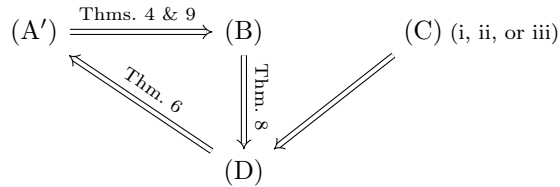

Irreducible compactifications of  widely-connected spaces are necessarily irreducible between points of their remainders (Corollary \ref{7}). On the other hand, there is a widely-connected $\widetilde  W\subseteq\mathbb R ^3$ about which every compactification is reducible.  That is, every compactification of $\widetilde W$ is reducible between every two points of $ \widetilde  W$. The example is presented near the end of Section \ref{peo}. It provides a negative answer to all parts of Question (C), improving one of our previous results from Section 4 of  \cite{lip}.  

We conclude Section \ref{peo} by noting that $\widetilde  W$ is not a counterexample to Bellamy's conjecture (a positive answer to Question (B)) because it densely embeds into an indecomposable subcontinuum of $[0,1]^3$.

\setstretch{1.1}
\section{Results for Tychonoff spaces}\label{people}

Let us begin by defining an  elementary property which we call strong indecomposability.  A topological space $X$ is \textit{strongly indecomposable} means that for every two non-empty disjoint open sets $U$ and $V$ there are two closed sets $A$ and $B$ such that: 
\begin{align}
&X=A\cup B; \\
&A\cap U\neq\varnothing; \\
&B\cap U\neq\varnothing;\text{ and} \\
&A\cap B\subseteq V.
\end{align}

Strongly indecomposable spaces are indecomposable.  For  if $C\subseteq X$ is neither dense nor nowhere dense in $X$, then applying strong indecomposability to the open sets $(\overline{C})\mkern-1mu{\vrule width0pt height 2ex}^\mathrm{o}$ and $X\setminus \overline C$ will show  $C$ is not connected.   

Every perfect totally disconnected space is strongly indecomposable, but the corresponding statement is false  if  \textit{totally disconnected} is  weakened to \textit{hereditarily disconnected} (see Example \ref{e2} in Section \ref{peo}). As a result,   indecomposability and strong indecomposability are not equivalent.  They are, however, equivalent for compact Hausdorff spaces. Moreover, Theorem \ref{4}  says  $\beta X$ is indecomposable if and only if $X$ is strongly indecomposable.  We assume, of course, that $X$ is Tychonoff. 

Before proving Theorem \ref{4} we need three rather basic lemmas.  The first is standard and will be used several times throughout the paper.

\begin{ul}[Theorem 6.1.23 in \cite{eng}]\label{1}In a compact Hausdorff space $X$ the connected component of a point $x\in X$ coincides with the quasi-component of $x$.\end{ul}

\begin{ul}Indecomposable compact Hausdorff spaces are strongly indecomposable.\label{2}\end{ul}

\begin{proof}Let $X$ be a compact Hausdorff space, and let $U$ and $V$ be non-empty disjoint open subsets of $X$.   Assuming $X$ is indecomposable, $U$ is not contained in a component of $X\setminus V$. By Lemma \ref{1}, $X\setminus V$ can be written as $A'\cup B'$ with  $A'$ and $B'$ closed and disjoint, and such that $A' \cap U$ and $B'\cap U$ are non-empty. By normality there are disjoint open sets $G$ and $H$ such that $A'\subseteq G$ and $B'\subseteq H$. Then conditions 2.1 through 2.4 are satisfied with $A:=X\setminus H$ and $B:=X\setminus G$.\end{proof}

\begin{ul}\label{3}Strong indecomposability is hereditary for dense subsets.\end{ul}

\begin{proof}Suppose $Y$ is strongly indecomposable and $X\subseteq Y=\overline X$. We prove $X$ is strongly indecomposable.  To that end, let $U$ and $V$ be non-empty disjoint open subsets of $X$. We find two $X$-closed sets $A$ and $B$  such that  conditions 2.1 through 2.4 are met. Well, there are two $Y$-open sets $U'$ and $V'$ such that  $U=U'\cap X$ and $V=V'\cap X$. Since $\overline X=Y$, we have $U'\cap V'=\varnothing$. By the assumption that $Y$ is strongly indecomposable, in $Y$ there are two closed sets $A'$ and $B'$ such that 
$$Y=A'\cup B'\; \text{, }\;A'\cap U'\neq\varnothing\;\text{, }\;B'\cap U'\neq\varnothing\;\text{,\; and }A'\cap B'\subseteq V'.$$ 
Apparently, $A:=A'\cap X$ and $B:=B'\cap X$ are  closed in $X$  and 2.1 and 2.4 hold.  Since $A'\cap U'=U'\setminus B'$ and $ B'\cap U'=U'\setminus A'$  are non-empty open subsets of $Y$,  by $\overline X=Y$ we have 2.2 and 2.3.\end{proof}

\begin{ut}\label{4}$\beta X$ is indecomposable if and only if $X$ is strongly indecomposable.\footnote{This intentionally resembles ``$\beta X$ is zero-dimensional if and only if $X$ is strongly zero-dimensional''  -- Theorem 6.2.12 in \cite{eng}.}\end{ut}

\begin{proof}If $\beta X$ is indecomposable then by Lemmas \ref{2} and \ref{3}, $X$ is strongly indecomposable. Now suppose $X$ is strongly indecomposable.  Let $K$ be a proper closed subset of $\beta X$ with non-empty interior. We show $K$ is not connected. 

By regularity of $\beta X$ there is a non-empty $\beta X$-open set $T$ such that $K\cap \cl_{\beta X} T=\varnothing$.  Applying  strong indecomposability to the $X$-open sets $\text{int}_X (K\cap X)$ and $T\cap X$   shows there are disjoint $X$-closed sets $A$ and $B$ such that $X\setminus T=A\cup B$, $A\cap K\neq\varnothing$, and $B\cap K\neq\varnothing$.   Observe that $$U:=A\setminus \cl_{\beta X} T\;\;\text{ and }\;\; V:=B\setminus \cl_{\beta X} T$$ are disjoint open subsets of $X$ each intersecting $K$. 

Put $W=\beta X\setminus \cl_{\beta X} T$.  Then  $W\cap X=U\cup V$, so density of $X$ in $\beta X$ implies  $W\subseteq \cl_{\beta X} U \cup \cl_{\beta X} V$.   All things considered,  $$W_0:=W\cap \cl_{\beta X} U\;\;\text{ and }\;\;W_1:=W\cap \cl_{\beta X} V$$ are relatively closed subsets of $W$, $K\subseteq W=W_0\cup W_1$, and $W_i\cap K\neq\varnothing$ for each $i< 2$. 

To complete the proof that $K$ is not connected, it  suffices to show $W_0\cap W_1=\varnothing$. Well, for a contradiction suppose there exists $p\in W_0\cap W_1$. By Urysohn's Lemma there is a mapping $F:\beta X \to [0,1]$ such that $F(p)=0$ and $F[\beta X \setminus W]=1$. Define $f:X\to [0,1]$ by 
	$$f(x)=\begin{cases}
	1 &\text{if } x\in U\\
	F(x) &\text{if } x\notin U.
	\end{cases}$$
If $\alpha$ is open in $[0,1]$, then 
	$$f^{-1}[\hspace{.5mm}\alpha\hspace{.5mm}]=
	\begin{cases}
	\big(F^{-1}[\hspace{.5mm}\alpha\hspace{.5mm}]\cap X\big)\cup U &\text{if } 1\in \alpha\\
	F^{-1}[\hspace{.5mm}\alpha\hspace{.5mm}]\cap V &\text{if } 1\notin\alpha,
	\end{cases}$$  
\noindent so $f$ is continuous.  Let $\beta f:\beta X\to [0,1]$ be the Stone-\v{C}ech extension of $f$. Since $p\in \cl_{\beta X} V$ and $f\restriction V=F\restriction V$, we have $\beta f(p)=F(p)=0$. On the other hand, $p\in \cl_{\beta X} U$ implies that $\beta f (p)=1$, a contradiction.
\end{proof}

\begin{uc}\label{5}$\beta X$ is an indecomposable continuum if and only if $X$ is strongly indecomposable and connected.  \end{uc}

\begin{proof}This follows from Theorem \ref{4} since $\beta X$ is a continuum if and only if $X$ is connected. \end{proof}

\begin{ut}\label{6}Let $X$ be a connected space. 
\begin{enumerate}[label=\textnormal{(\roman*)}]
\item If $X$ has  an irreducible compactification, then $\beta X$ is irreducible.
\item If $X$ has  an indecomposable compactification, then $\beta X$ is indecomposable.
\item If $X$ is indecomposable,  then every irreducible compactification of $X$ is indecomposable.\footnote{In other words, a decomposable compactification with three composants is not possible.}
\end{enumerate}

\end{ut}
\begin{proof}[Proof of (i)] Suppose $\gamma X$ is an irreducible compactification. Let $p$ and $q$ be such that $\gamma X$ is irreducible between them.  The Stone-\v{C}ech extension $\beta\iota:\beta X\to \gamma X$ of the inclusion $\iota:X\hookrightarrow \gamma X$  satisfies $\beta\iota[\beta X\setminus X]=\gamma X\setminus X$. As a result, $\beta\iota$ maps onto $\gamma X$, and maps proper subcontinua to proper subcontinua (in the words of \cite{bellll}, $\beta\iota$  maps $\beta X$  \textit{irreducibly} onto $\gamma X$).\footnote{Theorem \ref{6}(i) is really a special case of the Proposition before Corollary 4 in \cite{bellll}.}    So there exists $\langle p',q'\rangle \in \beta\iota^{-1}\{p\}\times \beta\iota^{-1}\{q\}$, and $\beta X$ is irreducible between $p'$ and $q'$.\end{proof}

\begin{proof}[Proof of (ii)]Suppose $X$ has an indecomposable compactification. By Lemmas \ref{2} and \ref{3}, $X$ is strongly indecomposable.  By Theorem \ref{4}, $\beta X$ is indecomposable.  Alternatively, if $\beta X$ is the union of two proper subcontinua $H$ and $K$, and $\gamma X$ is any compactification of $X$, then $\gamma X$   is the union of  proper subcontinua $\beta\iota[H]$ and $\beta\iota[K]$ ($\beta\iota$  from the proof of (i)).  That proves the contrapositive of (ii).  \end{proof}

\begin{proof}[Proof of (iii)]Suppose $X$ is indecomposable.  We prove every decomposable compactification of $X$ is reducible (this is the less awkward approach).   To that end, suppose  $\gamma X$ is a compactification which  decomposes into two proper subcontinua $H$ and $K$; $\gamma X=H\cup K$. Let $\langle p,q\rangle\in \gamma  X^2$. Our goal is to show  $\gamma X$ is reducible between $p$ and $q$, so we may clearly assume  $p\in H$ and $q\in K\setminus H$.  By indecomposability of $X$  there are non-empty disjoint $X$-open sets $U$ and $V$ such that $X\setminus H=U\cup V$.  Without loss of generality, $q\in \cl_{\gamma X} U$.  Observe that $H\cup U$ is connected since $H$ and $X$ are connected. The continuum $H\cup \cl_{\gamma X} U$ witnesses that $\gamma X$ is reducible between $p$ and $q$.   Since $p$ and $q$ were arbitrary, our proof is complete.
\end{proof}


\begin{uc}\label{7}If $\gamma W$ is an irreducible compactification of a widely-connected space $W$, then $\gamma W$ is irreducible between two points of the remainder $\gamma W \setminus W$.\end{uc}

\begin{proof}By hypothesis $\gamma W$ has two composants $P\neq Q$. Each composant contains a non-degenerate proper subcontinuum of $\gamma W$ (cf. `Boundary Bumping' Lemma 6.1.25 in \cite{eng}), so by the widely-connected property of $W$ there exists $p\in P\setminus W$ and $q\in Q\setminus W$. By Theorem \ref{6}(iii)  $\gamma W$ is indecomposable, so $P\cap Q=\varnothing$ and $\gamma W$ is irreducible between $p$ and $q$.\end{proof}

\begin{ur}Typically,  $\beta W$ is indecomposable when $W$ is widely-connected.  This is due to the fact that most widely-connected sets are constructed as dense subsets of indecomposable (metric) continua.    Gary Gruenhage  \cite{gru}  constructed  completely regular and perfectly normal examples by more technical set-theoretic methods, assuming   Martin's Axiom and the Continuum Hypothesis,  respectively.  Both of his examples are strongly indecomposable,   so their Stone-\v{C}ech compactifications are indecomposable.  Moreover, all of their co-infinite subsets are totally disconnected (it's worth noting that there is no connected \textit{metric} space with this property). So by the proof of Theorem \ref{4}, for  Gruenhage's $W$ we get $\abs{W\cap P}\leq 1$ for every composant $P$ of $\beta W$ (i.e. $\beta W$ is irreducible between every two points of $W$). This property is shared by Paul Swingle's original widely-connected sets in  \cite{swi}.\end{ur}

\section{Results for separable metrizable spaces}\label{lop}

Throughout this section,  $X$ is separable and metrizable.  

\begin{ut}\label{8}Every indecomposable connected compactification of  $X$ is irreducible.\footnote{There is a  non-metrizable indecomposable continuum with only one composant \cite{bell}, so the assumptions about $X$  are critical.}\end{ut}

\begin{proof}Suppose $\gamma X$ an indecomposable connected compactification. Let $\{U_n:n<\omega\}$ be a basis for $X$ consisting of non-empty open sets. For each $n<\omega$ let $U'_n$ be open in
$\gamma X$ such that $U'_n\cap X=U_n$. 

The collection $\{U'_n:n<\omega\}$ is a countable network for $X$ in $\gamma X$. That is, for every $x\in X$ and $\gamma X$-open  $W\ni x$  there exists $n<\omega$ such that $x\in U'_n\subseteq W$.  Indeed, by regularity there is a $\gamma X$-open set $V$ such that $x\in V\subseteq \cl_{\gamma X}V\subseteq W$. There exists $n<\omega$ such that $x\in U_n\subseteq V\cap X$.  Then by density of $X$ in $\gamma X$ we have $x\in U'_n\subseteq \cl_{\gamma X}U_n\subseteq \cl_{\gamma X}V\subseteq W$.   

To show that $\gamma X$ has (at least) two disjoint composants it suffices to show that every composant is a first category $F_\sigma$-set.  
Let $P$ be the composant of a point $p\in\gamma X$. For each $n<\omega$ such that $p\notin U'_n$, let $P_n$ be the component of $p$ in  $\gamma X\setminus U'_n$. If $p\in U'_n$ then set $P_n=\varnothing$. Obviously $\bigcup \{P_n:n<\omega\} \subseteq P$. On the other hand, if $K$ is a proper subcontinuum of $\gamma X$ with $p\in K$, then  there exists $n<\omega$ such that $U'_n\subseteq \gamma X\setminus K$.  Then $K\subseteq P_n$.  So $P\subseteq \bigcup \{P_n:n<\omega\}$.  Combining the two inclusions, we have $P= \bigcup \{P_n:n<\omega\}$. Each $P_n$ is closed and nowhere dense by indecomposability of $\gamma X$, thus $P$ is an $F_\sigma$-set of the first category.   \end{proof}

\begin{ur}Combining Theorems \ref{6}(iii) and  \ref{8}, we find that indecomposability and irreducibility are equivalent in compactifications of  widely-connected separable metric spaces.  For example, the Hilbert cube $[0,1]^\omega$, which is  a canonical compactification for widely-connected separable metric spaces,\footnote{The main result of   \cite{bow} is that every nowhere compact separable metric  space densely embeds into the Hilbert space $\ell ^2\simeq (0,1)^\omega$.  Widely-connected Hausdorff spaces have no compact neighborhoods (applying Theorem  6.2.9 in \cite{eng} to a compact neighborhood $N\neq W$ would show $W$ is not connected). So every widely-connected separable metric space has a compactification equal to $[0,1]^\omega$.} satisfies neither condition.   There is also dense widely-connected subset of the plane \cite{swi3} which, naturally,  has a compactification homeomorphic to $[0,1]^2$ (another decomposable continuum with only one composant).\end{ur}

\begin{ut}\label{9}If $X$ is strongly indecomposable, then $X$ has a metrizable indecomposable compactification.
\end{ut}

\begin{proof}Let  $\{U_i:i<\omega\}$ be a basis for $X$ consisting of non-empty open sets.  We apply strong indecomposability to each ordered pair of disjoint basic sets. Let $\Pi=\{\langle i,j\rangle\in\omega^2:U_i\cap U_j=\varnothing\}$. For each  $\langle i,j\rangle\in \Pi$ there are disjoint closed sets $A_{\langle i,j\rangle}$ and $B_{\langle i,j\rangle}$ such that $X\setminus U_i=A_{\langle i,j\rangle}\cup B_{\langle i,j\rangle}$, $A_{\langle i,j\rangle}\cap U_j\neq\varnothing$, and $B_{\langle i,j\rangle}\cap U_j\neq\varnothing$. 

By \cite{kur} \S44 V Corollary 4a and  \cite{kur} \S44 VI Lemma, 
$$\;\;\;\;\;\;\Big\{g\in \big([0,1]^\omega\big)^X:\overline {g[A_{\langle i,j\rangle}]}\cap \overline{g[B_{\langle i,j\rangle}]}=\varnothing\Big\}\;\;\;\;\;\;\big(\langle i,j\rangle\in \Pi\big)$$ is a dense open subset of the function space $([0,1]^\omega)^X$.\footnote{$([0,1]^\omega)^X$ is the set of continuous mappings from $X$ into the Hilbert cube $([0,1]^\omega,\rho)$, endowed with the complete metric $\varrho(f,g)=\sup \{\rho(f(x),g(x)):x\in X\}.$ The Lemma from  \S44 of \cite{kur} is formulated for compact $X$, but one easily sees that compactness is not needed for its proof.}  Now by Theorem 2 of  \cite{kur} \S44 VI, there is a homeomorphic embedding $h:X\hookrightarrow [0,1]^\omega$ such that $\overline{h[A_{\langle i,j\rangle}]}\cap \overline{h[B_{\langle i,j\rangle}]}=\varnothing$ for each $\langle i,j \rangle\in \Pi$. 

The metric compactification $\overline {h[X]}$ is indecomposable. For if $K$ is a proper closed subset of $\overline {h[X]}$ with non-empty interior, then there exists $\langle i,j\rangle\in \Pi$ such that $\overline {h[U_i]}\cap K=\varnothing$ and $h[U_j]\subseteq K$.  Then $h[A_{\langle i,j\rangle}]\cap K\neq\varnothing$ and $h[B_{\langle i,j\rangle}]\cap K\neq\varnothing$. As $\overline {h[X]}=\overline {h[U_i]}\cup \overline {h[A_{\langle i,j\rangle}]}\cup \overline {h[B_{\langle i,j\rangle}]}$, we have $K\subseteq \overline{h[A_{\langle i,j\rangle}]}\cup \overline {h[B_{\langle i,j\rangle}]}$. Finally, $\overline{h[A_{\langle i,j\rangle}]}\cap \overline {h[B_{\langle i,j\rangle}]}=\varnothing$, so $K$ is not connected.\end{proof}

By Theorems \ref{4}  through \ref{9}, we have the following.

\begin{uc}\label{10}If  $X$ is indecomposable and connected, then the following are equivalent:
\begin{enumerate}[label=\textnormal{(\roman*)}]	
\item  $\beta X$ is indecomposable; 
\item $X$ is strongly indecomposable;
\item  $X$ has a metrizable indecomposable compactification;
\item $X$ has an indecomposable compactification;
\item $X$ has an irreducible compactification;
\item $\beta X$ is irreducible.
\end{enumerate}
\end{uc}

\begin{ur}Questions (A$'$), (B), and (D) are equivalent by  (i)$\Leftrightarrow$(iii)$\Leftrightarrow$(v). Finally, if Question (C) has a positive answer for a particular $W$, as witnessed by some compactification $\gamma W$, then $\gamma W$ obviously has more than one composant. This explains the implication (C)$\implies$(D) in Figure 1.\end{ur}

\section{Examples in Euclidean 3-space} \label{peo}

We have already seen that the compactifying process can destroy indecomposability.  In fact, a single limit point can turn a widely-connected set into a decomposable one. This is the subject of Example 1.

\begin{ue}\label{e1}There is a widely-connected subset of Euclidean $3$-space which fails to be indecomposable upon the addition of one  limit point.\footnote{Mary Ellen Rudin hinted at  such an example in \cite{rud3}, where the following was shown: If $I\subseteq \mathbb R ^2$ is an indecomposable connected set, and $p\in \mathbb R ^2$ is a limit point of $I$, then $I\cup \{p\}$ is also indecomposable.  Our example shows that $\mathbb R ^2$ cannot be replaced with $\mathbb R ^3$.
And, although $q[W]$ embeds into the plane, its one-point augmentation $q[W]\cup\{ \langle 0,0,0\rangle\}$ does not.}\end{ue}

\noindent\textit{Construction.} Figure \ref{fiffe} depicts a rectilinear version of the bucket-handle continuum $K\subseteq [0,1]^2$, together with the diagonal $\Delta:=\{\langle x,x\rangle:x\in[0,1]\}$. The set $K\setminus \Delta$ is the union of two open sets $K_0:=\{\langle x,y\rangle\in K:x<y\}$ and $K_1:=\{\langle x,y\rangle\in K:x>y\}$.
\begin{wrapfigure}[16]{R}{0.45\textwidth}
\centering
  \includegraphics[scale=0.38]{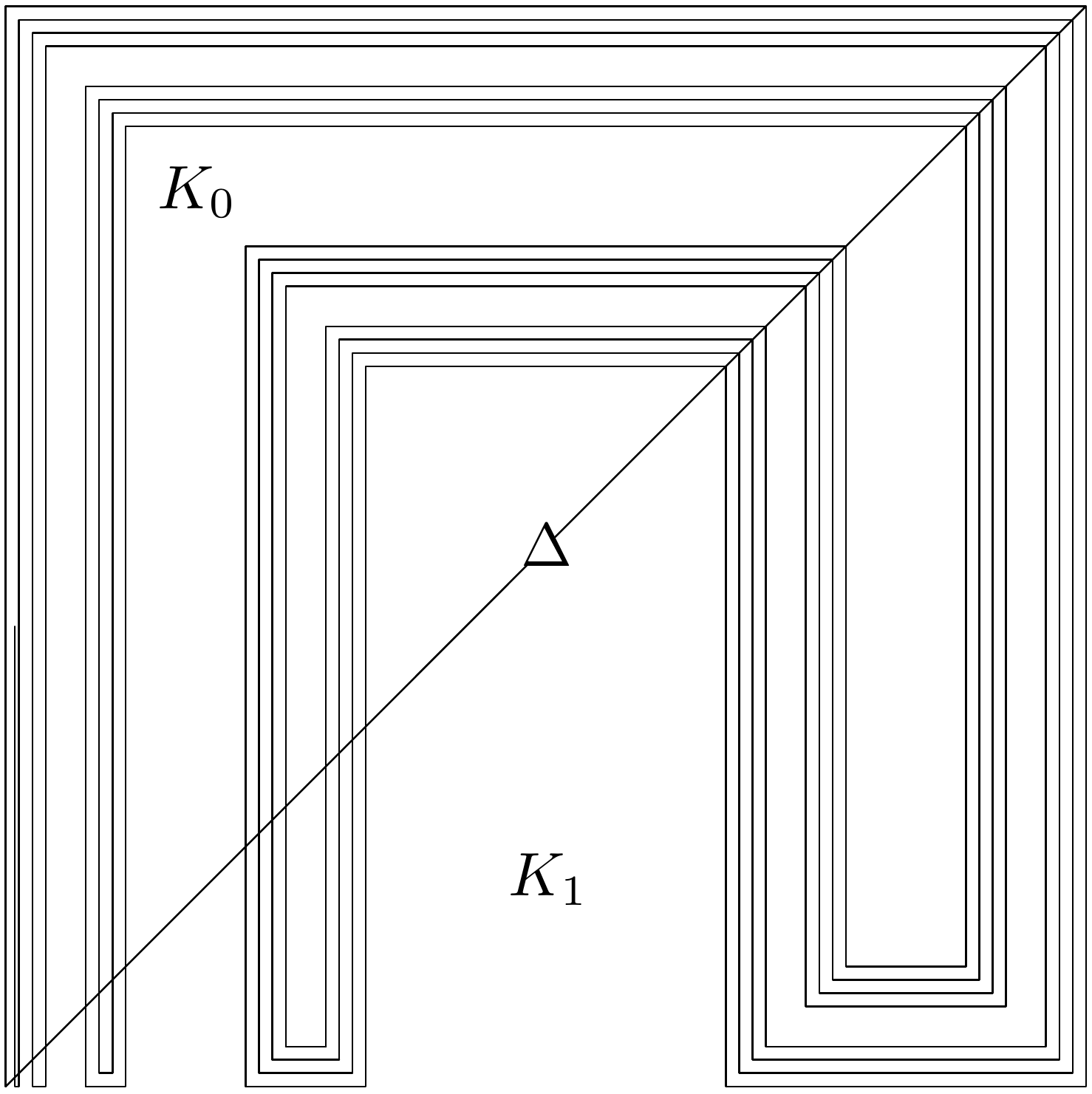} 
  \caption{$K\setminus \Delta=K_0\cup K_1$}
  \label{fiffe}
\end{wrapfigure}

Wojciech D\k{e}bski \cite{deb} described a closed set $A\subsetneq K$ such that $K\setminus A$ is connected and $A$ intersects every composant of $K$. In \cite{lip}, we constructed a  widely-connected  $W\subseteq K$  by deleting a countable infinity of D\k{e}bski sets (copies of $A$) from $K$.   By the particular construction of $W$,  we can assume $A\subseteq K_0\setminus W$. 

Consider $\widehat W:=W\cup A$ to be on the surface of the unit sphere  $\mathbb S^2\subseteq (\mathbb R ^3,d)$.  Let $q:\widehat W\to \mathbb R ^3$ be the  mapping defined by the scalar multiplication $q(x)=d(x,A) \cdot x$, where $d(x,A)=\inf\{d(x,y):y\in A\}$.  So $q$  shrinks $A$ to the single point at the  origin $\langle 0,0,0\rangle$.  Meanwhile, compactness of $A$ implies that $q\restriction W$ is a homeomorphism. In particular, $q[W]$ is widely-connected.

\vspace{2mm}

\noindent\textit{Verification.} By the constructions of $A$  and $W$ (see Section 6 of \cite{lip}), $A$ intersects every quasi-component of $\widehat W\cap K_0$. So $q[\widehat W\cap K_0]$ is connected. The set $q[\widehat W\cap K_0]$ is neither dense nor nowhere dense in $ q[\widehat W]$, so $q[\widehat W]=q[W]\cup \{\langle 0,0,0\rangle\}$ is not indecomposable. 
 
 \begin{ur}Likewise, the one-point compactification of an indecomposable connected space (e.g. $K\setminus A$) may be decomposable.\end{ur}
 
 \
 
 The reader may easily verify that every widely connected but not strongly indecomposable space contains a perfect hereditarily disconnected set that is not strongly indecomposable. Example \ref{e2} exhibits a very concrete set with these properties.
 
\begin{ue} \label{e2}There is a  perfect hereditarily disconnected set which is not strongly indecomposable.
\end{ue}

\begin{wrapfigure}[11]{R}{0.47\textwidth}
\centering
  \includegraphics[scale=.28]{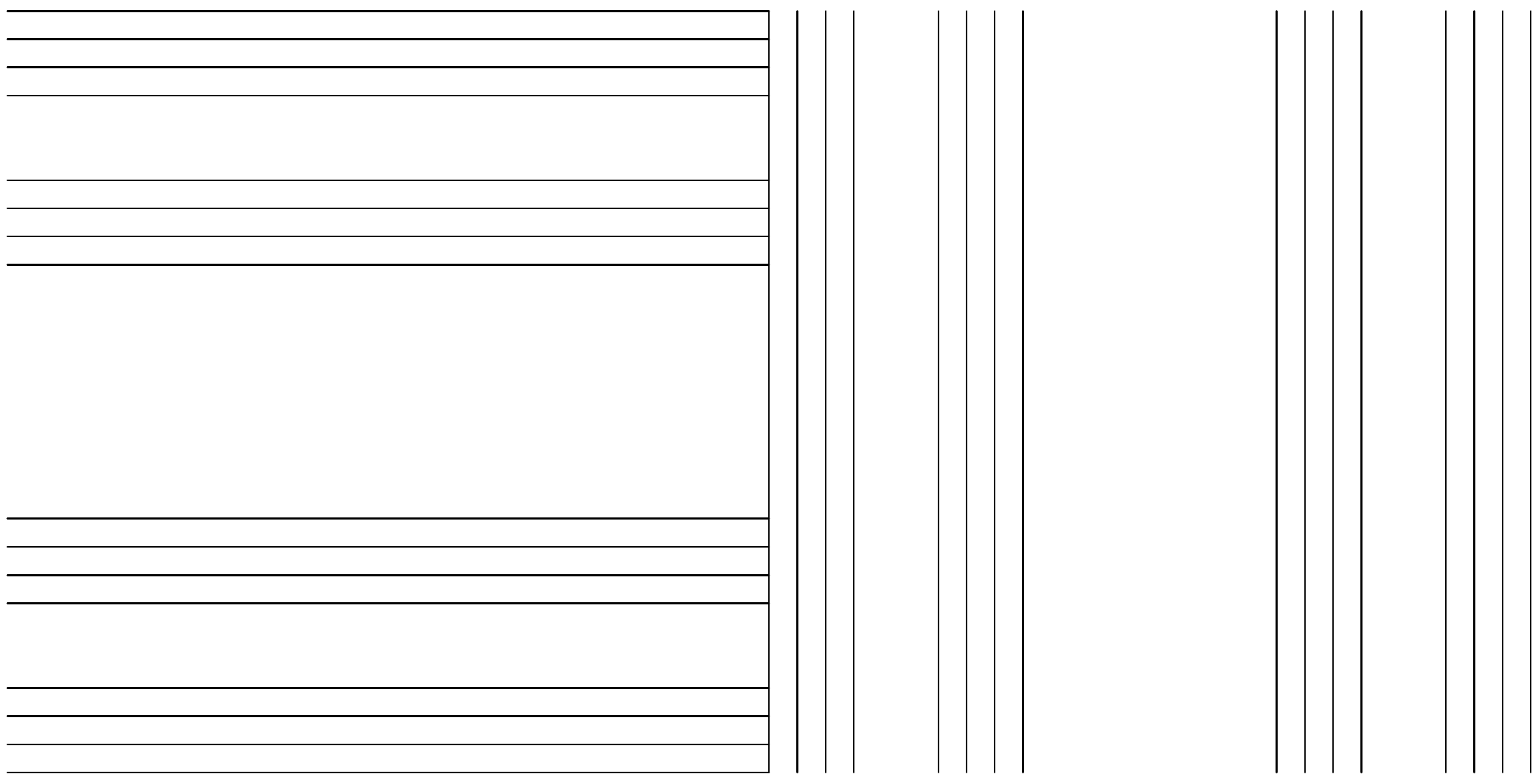} 
   \caption{$Y\subseteq \big([-1,0]\times C\big)\cup \big(C\times [0,1]\big)$}
  \label{fif}
\end{wrapfigure}

\noindent\textit{Construction \& Verification.} Let $C$ be the middle-thirds Cantor set in $[0,1]$.  If $\mathbb Q$ and $\mathbb P$ denote the rationals and irrationals, respectively, then let $X=(C'\times \mathbb Q)\cup (C''\times \mathbb P)$, where $C'$ is the set of endpoints of intervals removed during the construction of $C$, and $C''=C\setminus C'$.  So $X$ is the Knaster-Kuratowski fan without its dispersion point.   Note  that $\overline{X\cap (\{c\}\times \mathbb R)}=\{c\}\times \mathbb R$ for each $c\in C$, and $\{X\cap (\{c\}\times \mathbb R):c\in C\}$ is the set of quasi-components of $X$.

Let $X_1=X\cap (\mathbb R \times [0,1])$, and let $X_0=\{\langle -y,x\rangle:\langle x,y\rangle\in X_1\}$ be the copy of  $X_1$ rotated ninety degrees about the origin. 
The set $Y:=X_0\cup X_1$ is certainly perfect and hereditarily disconnected. The two properties of $X$ noted above imply that $X_0$  is contained in a quasi-component of $Y\setminus((\sfrac{1}{2},1]\times [0,1])$, therefore  $Y$ is not strongly indecomposable.


\

Lemma 9.8 in \cite{KH} says that an open $U\subseteq \beta X$ is connected if and only if $U\cap X$ is connected (this was actually  demonstrated in the proof of Theorem \ref{4}). Thus, if $X$ is an indecomposable connected set then $\beta X$ is a continuum each of whose connected open subsets is dense.   

\begin{ue}\label{e3}There is a decomposable continuum each of whose connected open subsets is dense.   \end{ue}

\begin{wrapfigure}[18]{R}{0.4\textwidth}
\centering
 \includegraphics[scale=.37]{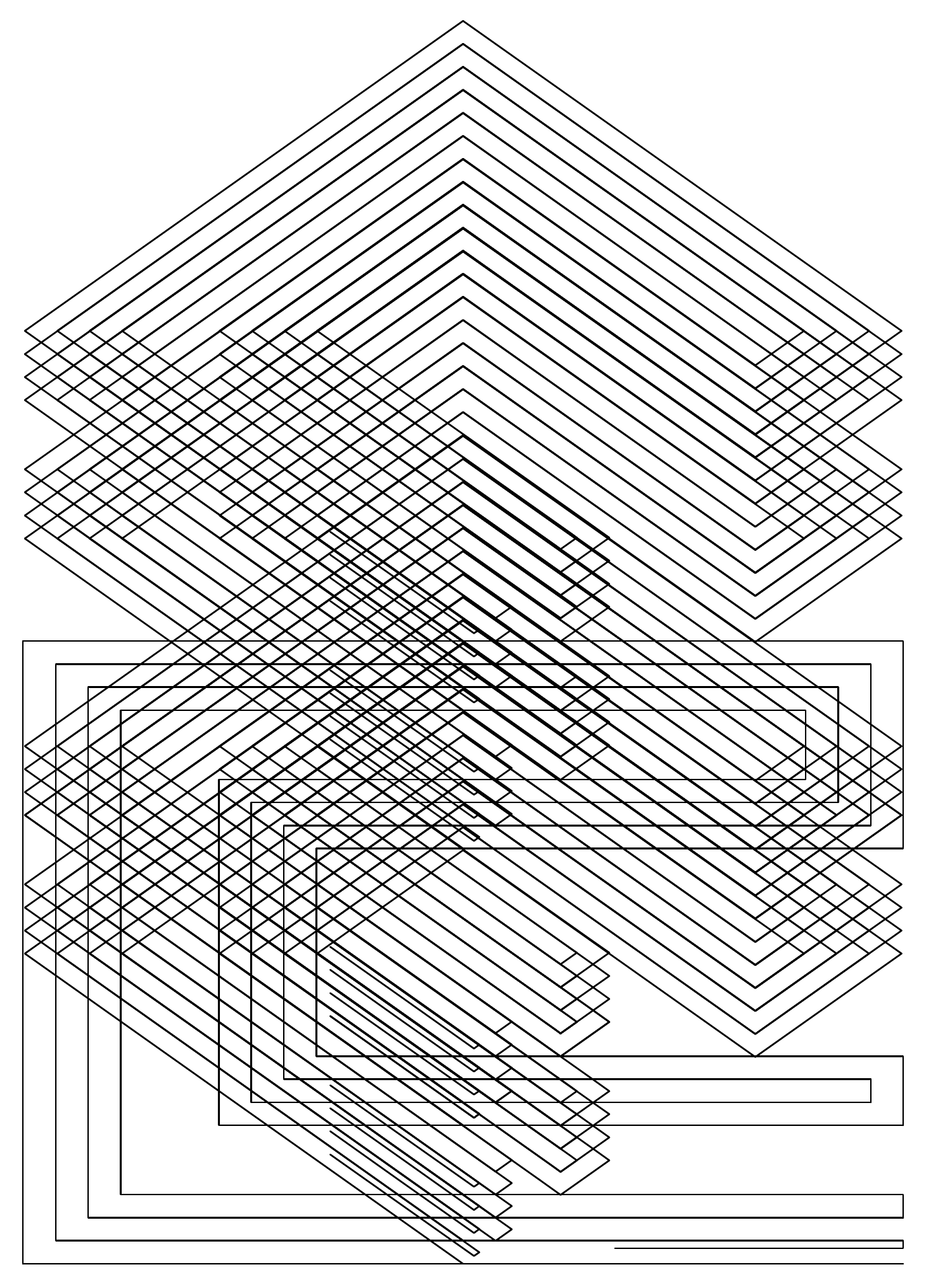}
   \caption{$D\subseteq \mathbb R ^3$}
  \label{fiff}
\end{wrapfigure}

\noindent\textit{Construction. }Let $K$ and $C$ be as in Examples \ref{e1} and \ref{e2}, respectively. Let $D$ be the quotient of the disjoint union $$[\{0\}\times K]\cup [\{1\}\times C\times K]$$ obtained by identifying  $\langle 0,\langle c,\sfrac{1}{2}\rangle \rangle$ with $\langle 1,c,\langle0,0\rangle\rangle$ for each $c\in C$. An embedding of $D$  is depicted in Figure \ref{fiff}. 


\vspace{2mm}

\noindent\textit{Verification.} Since $\{0\}\times K$ is a proper subcontinuum of $D$ with non-void interior,  $D$ is decomposable. 

Now suppose $G$ is a non-void open subset of $D$ such that $\overline G\neq D$. Then $D\setminus \overline G$ is non-empty and open, so at least one of the following open sets is also non-empty.
\begin{align*}
V_0:&=[\{0\}\times (K\setminus (C\times \{\sfrac{1}{2}\}))]\setminus \overline G\text{; and} \\
V_1:&=[\{1\}\times C \times (K\setminus \langle 0,0\rangle)]\setminus \overline G.
\end{align*}
We show $G$ is not connected in each of the two cases $V_0\neq\varnothing$ and $V_1\neq\varnothing$. Throughout, we assume $G\cap  (\{0\}\times K)\neq\varnothing$ (otherwise it is clear that $G$ is not connected). 

\begin{uca1} $V_0\neq\varnothing$. \end{uca1}

By strong indecomposability of $\{0\}\times K$ there are disjoint closed sets $A$ and $B$ such that $(\{0\}\times K)\setminus V_0=A\cup B$, $A\cap G\neq\varnothing$, and $B\cap G\neq\varnothing$.  Let $A'=A\cup [\{1\}\times \{c\in C:\langle c,\sfrac{1}{2}\rangle\in A\}\times K]$ (with appropriate points identified), and define $B'$ similarly.  Then  $\{A',B'\}$ is a clopen partition of $D\setminus V_0\supseteq G$.   Each partition set  intersects $G$, so $G$ is not connected.

\begin{uca2}$V_1\neq\varnothing$. \end{uca2}

We can  assume $G$ contains a dense subset of $\{0\}\times(C\times \{\sfrac{1}{2}\})$.  Otherwise 
we could re-define $V_0$ to be a non-void relatively open subset of $[\{0\}\times (C\times \{\sfrac{1}{2}\})]\setminus G$, and use Case 1 methods to get a separation of $G$.   For although  $V_0$ (re-defined) is no longer open in $\{0\}\times K$, the separated sets $A$ and $B$ can be constructed as follows.  There is an arc $\alpha\subseteq \{0\}\times K$ such that $\text{ends}(\alpha)\subseteq V_0$ and $\alpha\cap G\neq\varnothing$.   There is then an epsilon-wide tube $\tau\subseteq \{0\}\times K$ (meaning a very thin $\tau\simeq C\times [0,1]$) such that $\alpha\subseteq \tau$, $\partial \tau\subseteq V_0$, and $G\cap(\{0\}\times K)\setminus \tau\neq\varnothing$. Put $A=\tau\setminus V_0$ and $B=(\{0\}\times K)\setminus (\tau\cup V_0)$.  Enlarging $A$ and $B$ to  $A'$ and $B'$ will again show $G$ is not connected.  

Let $S\times T$  be a non-empty  $C\times K$-open set such that $\{1\}\times S\times T\subseteq V_1$. 

Since $G$ is open and contains a dense subset of $\{0\}\times(C\times \{\sfrac{1}{2}\})$, there is a $K$-open set $U$  and a non-empty  $C$-clopen $E\subseteq S$ such that  $\langle 0,0\rangle\in U$ and $\{1\}\times E\times U\subseteq G$. By strong indecomposability of $K$, there is a clopen partition $\{A,B\}$ of $K\setminus T$ such that $A\cap U\neq\varnothing$ and $B\cap U\neq\varnothing$. Without loss of generality,  $\langle 0,0\rangle\in A$. Then $G$ is contained in the two disjoint closed sets $\{1\}\times E\times B$ and  (the quotient of) $$\big(\{0\}\times K\big)\cup \big(\{1\}\times E\times A\big)\cup \big(\{1\}\times (C\setminus E)\times K\big).$$ Each  set intersects $G$, so $G$ is not connected. This concludes Example \ref{e3}.
  
\setstretch{1.1}

\

We end with a counterexample to Question (C), as promised in Section 1.2.  

\begin{ue}\label{e4}There is a widely-connected set about which every compactification is reducible.\end{ue}

\noindent\textit{Construction Phase I.} 
Let $K$, $\Delta$, $K_0$, and $K_1$ be as defined in Example \ref{e1}. Let $X$ be any widely-connected subset of $K$ with more than one point (Example \ref{e1} or  \cite{swi,lip}). For each $i<2$  (that is, $i\in \{0,1\}$) put $X_i '=(X\cap K_i)\cup \Delta '$, where  $\Delta '$ is the diagonal Cantor set $\Delta\cap K$.  Note that each $X'_i$ is hereditarily disconnected, but $X':=X'_0\cup X'_1=X\cup \Delta '$ is connected since $X$ is connected and (necessarily) dense in $K$.  
 
\begin{ucl1}\label{lemmi}If $A_0$ and $A_1$ are relatively clopen subsets of $X'_0$ and $X'_1$, respectively, and $A_0\cap \Delta=A_1\cap \Delta$, then $A_0\cup A_1$ is clopen in the topology of $X'$.\end{ucl1}

\begin{proof}Clearly each $X_i '$ is closed in $X'$ and each $X'_i\setminus \Delta$ is open in $X'$.  So $A_0\cup A_1$ is closed in $X'$ and $(A_0\cup A_1) \setminus \Delta$ is open in $X'$.  It remains to show that  $A_0\cup A_1$ is an $X'$-neighborhood of $(A_0\cup A_1) \cap \Delta$. Well, let  $U_0$ and $U_1$ be $X'$-open sets such that $U_i\cap X'_i =A_i$. The  hypothesis $A_0\cap \Delta=A_1\cap \Delta$  implies  $$(A_0\cup A_1) \cap \Delta\subseteq A_0\cap A_1\subseteq  U_0\cap U_1=(U_0\cap U_1)\cap (X'_0\cup X'_1)\subseteq (U_0\cap X'_0)\cup(U_1\cap X'_1)=A_0\cup A_1;$$ so $U_0\cap U_1$ witnesses that $(A_0\cup A_1) \cap \Delta$ is contained in the $X'$-interior of $A_0\cup A_1$.\end{proof}

\begin{ucl2}\label{pop}For each $i<2$,   $\{0\}\times X'_i$ is a quasi-component of the set $$Y_i:=(\{0\}\times X'_i)\cup(\{1/n:n=1,2,3,...\}\times X'_{1-i})$$ in the subspace topology inherited from $[0,1] \times X'$. \end{ucl2}

\begin{proof}Fix $i<2$.  Clearly $\{0\}\times X'_i$ contains a quasi-component of $Y_i$. We need to prove that  $\{0\}\times X'_i$ is contained in a quasi-component of $Y_i$. To that end, let $A$ be a clopen subset of $Y_i$ such that $A\cap (\{0\}\times X'_i)\neq\varnothing$. We show $\{0\}\times X'_i\subseteq A$. 

For each $a\in A\cap (\{0\}\times \Delta)$ there is an integer $n(a)>0$ and an $X'$-open set $U(a)$ such that $a\in ([0,1/n(a)]\times U(a))\cap Y_i\subseteq A$. By compactness of $A\cap (\{0\}\times \Delta)$ there is a finite $\dot A\subseteq A$ such that $$A\cap (\{0\}\times \Delta)\subseteq \bigcup\big\{[0,1/n(a)]\times U(a):a\in \dot A\big\}.$$ Similarly, there is a finite $\dot B\subseteq Y_i\setminus A$, a set of positive integers $\{n(b):b\in \dot B\}$, and a collection of $X '$-open sets $\{U(b) :b\in \dot B\}$ such that $$(Y_i\setminus A)\cap (\{0\}\times \Delta)\subseteq Y_i\cap \bigcup \big\{[0,1/n(b)]\times U(b):b\in \dot B\big\}\subseteq Y_i\setminus A.$$

Put $m=\max\{n(y):y\in 	\dot A \cup \dot B\}$, and let 
\begin{align*}A_0&=A\cap (\{0\}\times X'_i)\text{; and}\\
A_1&=\{0\}\times \{x\in X'_{1-i}:\langle 1/m,x\rangle\in A\}.
\end{align*}
Then $A_0\cap (\{0\}\times \Delta)=A_1\cap (\{0\}\times \Delta)$. So $A_0\cup A_1$ is clopen in $\{0\}\times X'$ by Claim 1.  Since $\{0\}\times X'$  is connected and $A_0\neq\varnothing$, it follows that $A_0\cup A_1=\{0\}\times X'$. Thus $\{0\}\times X'_i\subseteq A_0\subseteq A$. \end{proof}

\begin{ur}Like $Y$ in Example \ref{e2}, $Y_i$ is perfect  and hereditarily disconnected but not strongly indecomposable. 
\end{ur}

\vspace{2mm}

\noindent\textit{Construction Phase II.} Roughly speaking, we will use a dense homeomorphism orbit in $2^\mathbb Z$ to form a linked chain of $X'_0$'s and $X'_1$'s.

Define $e\in 2^\mathbb Z$  by concatenating all finite binary sequences in the positive and negative  directions of $\mathbb Z$. That is, if $\{b_i:i<\omega\}$ enumerates $2^{<\omega}$, then put $e\restriction \mathbb Z ^+ =b_0 ^{\;\frown} b_1 ^{\;\frown} b_2 ^{\;\frown} ...$ and   $e(n)=e(-n)$ for $n\in \mathbb Z ^-$. Here $\mathbb Z^-=\{...,-2,-1,0\}$ and $\mathbb Z^+=\{0,1,2,...\}$.  

Let $\eta:2^\mathbb Z \to 2^\mathbb Z$ be the shift map $\eta(f)(n)=f(n+1)$. Clearly $\eta$ is a homeomorphism, and the backward and forward orbits 
$$E_0=\{\eta ^n(e): n\in \mathbb Z ^-\} \text{ \;and\; }  E_1=\{\eta ^n(e): n\in \mathbb Z ^+\setminus \{0\}\}$$
 are each dense in $2^\mathbb Z$. Also,  $E_0\cap E_1=\varnothing$. For if $n,m<\omega$ such that $\eta^{-n}(e)=\eta^m(e)$, then $\eta^{n+m}(e)=e$. Density of $E_1$ implies that $e$ is not a periodic point of $\eta$. So $n+m=0$, whence $n=m=0$.  We have $\eta^{-n}(e)=\eta^m(e)=e\in E_0\setminus E_1$. 

Let $p,q\in \Delta '$ be such that  $K$ is irreducible between $p$ and $q$ (take  $p=\langle 1,1\rangle$ and $q=\langle \sfrac{1}{4},\sfrac{1}{4}\rangle$, for example). Let $\sim\;\subseteq (2^\mathbb Z \times K)^2$ be the relation  $\{\langle \langle f,p\rangle,  \langle \eta(f),q\rangle\rangle:f\in 2^\mathbb Z \}$.  Define $$W=(E_0\times X'_0)\cup (E_1\times X'_1)\text{\; and\; }\widetilde{W}=W/\sim.$$
Note that $W$ is hereditarily disconnected since $E_0\cap E_1=\varnothing$ and each $X_i'$ is hereditarily disconnected.

 \begin{ucl3}$\widetilde{W}$ is connected.\end{ucl3}	
 \begin{proof}For each $i<2$ and $e'\in E_i$ we see that $\{e'\}\times X'_i$ is a quasi-component of $W$. This follows from Claim   2  since there is a sequence of points $(e'_n)\in (E_{1-i})^\omega$ such that $e'_n\to e'$ as $n\to\infty$. So the relation $\sim$ ties together the quasi-components of $W$ and the claim holds. 
\end{proof}

\begin{ucl4}$\widetilde{W}$ is \textit{widely}-connected.\end{ucl4}
\begin{proof}Let  $C$ be a non-empty connected subset of $\widetilde W$ such that $\overline C\neq \widetilde W$.  We show $\abs{C}=1$.  

Let $x\in C$, and let $U\times V$ be a non-empty open subset of $2^\mathbb Z\times K$ such that $\{p,q\}\cap V=C\cap (U\times V) =\varnothing$. There exists $x'\in W$ such that $x=x'/\sim$. Let $e'$ be the first coordinate of $x'$. Because $E_0$ and $E_1$ are dense in $2^\mathbb Z$,  there are integers $n,m>0$ such that $\{\eta ^{-n}(e'),\eta^m(e')\}\subseteq U$. Since $\eta ^{n+m}$ is a homeomorphism and $2^\mathbb Z$ has a basis of clopen sets, there is a clopen $A\subseteq 2^\mathbb Z$ such that $\eta ^{-n}(e')\in A\subseteq U$ and $\eta^{n+m}[A]\subseteq U$. 

Irreducibility of $K$ between $p$ and $q$ implies that there are  disjoint compact sets $L$ and $M$ such  that  $p\in L$,  $q\in M$, and $K\setminus V=L\cup M$. Observe that $$T:= \Big (\big[W\setminus (U\times V)\big]\cap \big[(A\times L)\cup (\eta ^{n+m}[A]\times M)\hspace{1.5mm}\cup \hspace{-1.5mm} \bigcup _{0<i<n+m\hspace{-1.5mm}} \hspace{-1mm}(\eta ^i [A]\times K)\big]\Big )\big/\sim$$ is a relatively clopen subset of $\widetilde W\setminus (U\times V)$. And  $x\in T$ since $e'\in \eta^n[A]$ and $0<n<n+m$. Thus $C\subseteq T$.

Further, let $R=T\cap \big(\big\{\eta^i(e'):-n\leq i\leq m\big\}\times K\big)/\sim$ and behold:  (i) $x\in R$, and (ii) if $y\in T\setminus R$ then by the construction of $T$ there is a relatively clopen $S\subseteq T$ such that $R\subseteq S\subseteq T\setminus \{y\}$.  Thus $C\subseteq R$.  But $R$ is hereditarily disconnected (see Figure \ref{366}), so  $C=\{x\}$.\end{proof}

\begin{ucl5}\label{imop} Every compactification of $\widetilde {W}$  is reducible between every two points of $\widetilde {W}$.

\end{ucl5}
\begin{proof} Let $\gamma \widetilde W$ be a compactification of $\widetilde W$, and let $\langle x,y\rangle \in \widetilde {W} ^{\raisebox{3pt}{\scriptsize \hspace{.3mm}2}}$\hspace{-.5mm}.

There are two points $x',y'\in W$ such that $x=x'/\sim$ and $y=y'/\sim$. Let $n,m\in\mathbb Z$ such that $\eta ^n(e)$ and $\eta ^m(e)$ are the first coordinates of $x'$ and $y'$, respectively. Assume  $n\leq m$.  There is a non-empty clopen $A\subseteq 2^\mathbb Z$ such that $A\cap \{\eta ^i(e):n-1\leq i\leq m+1\}=\varnothing$.   Set $T= (A\times K)/\sim$.

For $i\in\mathbb Z$, put $\delta(i)=0$ if $i\leq0$ and $\delta(i)=1$ if $i> 0$. Then $$Q:=\Big(\bigcup _{n\leq i\leq m} \big\{\eta ^i(e)\big\}\times W_{\delta(i)}\Big)\big /\sim$$  is a well-defined  subset of $\widetilde W\setminus T$ containing $x$ and $y$. Each fiber $\{\eta ^i(e)\}\times W_{\delta(i)}$, $n\leq i\leq m$, is a quasi-component of $W\setminus (A\times K)$ by density of $E_{1-\delta(i)}$ in $2^{\mathbb Z}\setminus A$ (Claim 2).  These fibers are tied together by $\sim$, so that  $Q$ is contained in a quasi-component of $\widetilde W \setminus T$.  

Let $T'$ be an  open subset of $\gamma \widetilde W$ such that $T'\cap \widetilde W=T$. Let $Q'$ be the quasi-component of $x$ in $\gamma \widetilde W \setminus T'$.  Evidently $\{x,y\}\subseteq Q\subseteq Q'$.  Further,  $Q'$ is a (proper) subcontinuum of $\gamma \widetilde W$ by Lemma \ref{1}, so $\gamma \widetilde W$ is reducible between $x$ and $y$. \end{proof}
\begin{figure}[h]
\centering
  \includegraphics[scale=0.6]{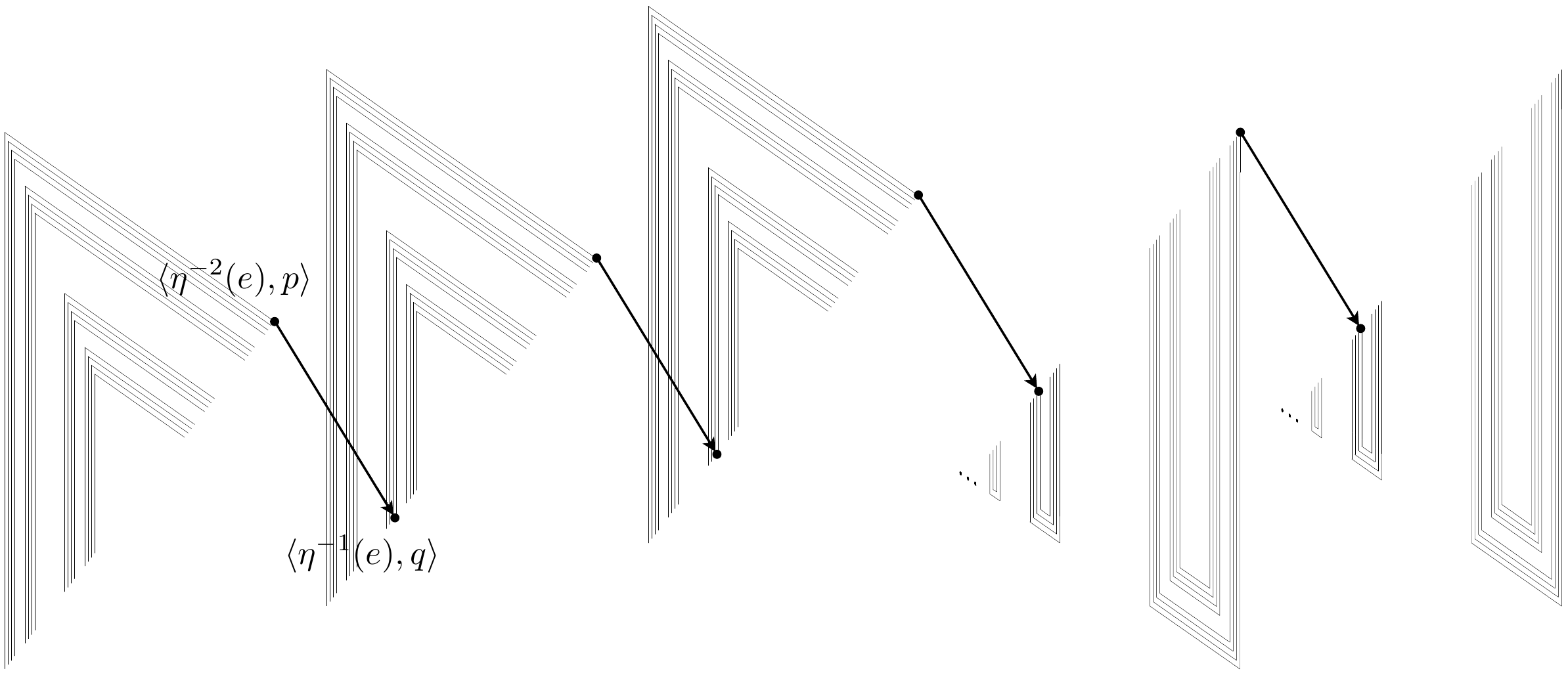} 
  \caption{Superset of $R$ if $n=m=2$ and $e'=e$}
  \label{366}
\end{figure}

\begin{urs}\

\noindent\begin{enumerate}

\item[(a)] The entire quotient $(2^\mathbb Z\times K)/\sim$ is Hausdorff ($\sim$ is closed), and is therefore metrizable.  Together with $\dim((2^\mathbb Z\times K)/\sim)=1$, this implies via the Menger-N\"{o}beling Theorem (\S45 VII Theorem 1 in \cite{kur}) 
 that $(2^\mathbb Z\times K)/\sim\;\supseteq \widetilde W$ embeds into Euclidean $3$-space.\footnote{
For a cube embedding of a space similar to $\widetilde W$,  see \cite{lip2}. } 
\item[(b)]  From Claim 5 it follows that for every compactification $\gamma \widetilde W$  there is a composant $P$ of $\gamma \widetilde W$ such that $\widetilde W \subseteq P$.   So the answer to Question (C) is \textit{no}. However, $\widetilde W$ is not a counterexample to Bellamy's conjecture.  It is a dense subset of   the metrizable continuum $(2^{\mathbb Z}\times K) /\sim$ which, by the arguments used to prove  Claim 4, is indecomposable.   \end{enumerate}
\end{urs}

\section{Related questions}

Towards a potential negative answer to Question (A), we would like to know:

\begin{uq}[E]Is there a widely-connected T$_1$ or T$_2$ topological space  which is not strongly indecomposable?\end{uq}

There are also several interesting variations of Questions (B)  and (D). For instance, let $X$ be a separable metrizable connected space. Does  $X$ densely embed into an indecomposable/irreducible continuum if $X$ is (i) locally compact indecomposable? (ii) an indecomposable plane set? (iii) an indecomposable one-to-one image of $[0,\infty)$ or $(-\infty,\infty)$?  Some tangential  results for sets of type (iii) are given by the author in \cite{lip3}, but these problems are currently unsolved.

\small


\begin{thebibliography}{HD}

\bibitem{belll}D. Bellamy,  A non-metric indecomposable continuum, Duke Math. J. 38 (1971) 15-20. 

\bibitem{bellll}\bysame,  Composants of Hausdorff indecomposable continua; a mapping approach, Pacific J. Math. 47 no. 2  (1973) 303-309.

\bibitem{bell}\bysame, Indecomposable continua with one and two composants, Fund. Math. 101 No. 2 (1978) 129-134. 

\bibitem{bel}\bysame, Questions in and out of context, Open Problems in Topology II, Elsevier (2007) 259-262.

\bibitem{bow} P. Bowers, Dense embeddings of nowhere locally compact separable metric spaces, Topology
Appl. 26 (1987) 1-12. 

\bibitem{deb}W. D\k{e}bski, Note on a problem by H. Cook. Houston J. Math. 17 (1991) 439-441.

\bibitem{dic} R. F. Dickman Jr., A necessary and sufficient condition for $\beta X\setminus X$ to be an indecomposable continuum, Proc. Amer. Math. Soc. 33 (1972), 191-194.

\bibitem{eng}R. Engelking, General Topology, Revised and completed edition Sigma Series in Pure Mathematics 6, Heldermann  Verlag, Berlin, 1989.

\bibitem{gru}G. Gruenhage, Spaces in Which the Nondegenerate Connected Sets Are the Cofinite Sets. Proceedings of the American Mathematical Society Vol. 122 No. 3 (1994) 911-924.

\bibitem{kur}K. Kuratowski, Topology Vol. II, Academic Press, 1968.

\bibitem{lip2}D. Lipham, Embedding a special irreducible set,  arXiv preprint     \url{https://arxiv.org/pdf/1804.05440.pdf}.

\bibitem{lip3}\bysame, One-to-one composant mappings of $[0,\infty)$ and $(-\infty,\infty)$   arXiv preprint.   \url{https://arxiv.org/pdf/1707.05007.pdf}.

\bibitem{lip}\bysame, Widely-connected sets in the bucket-handle continuum,  Fund. Math. 240 No. 2 (2018) 161-174.



\bibitem{rep}J. Mioduszewski, Spring 2004 Conference Report, Topology Proceedings, Volume 28 (2004) 301-326.

\bibitem{rud3}M.E. Rudin, A property of indecomposable connected sets, Proc. Amer. Math. Soc. 8 (1957) 1152-1157.

\bibitem{swi} P.M. Swingle, Two types of connected sets, Bull. Amer. Math. Soc. 37 (1931) 254-258.

\bibitem{swi2} \bysame, Indecomposable connexes, Bull. Amer. Math. Soc. 47 (1941) 796-803. 

\bibitem{swi3} \bysame, The closure of types of connected sets, Proc. Amer. Math. Soc. 2 (1951) 178-185.

\bibitem{KH} R.C. Walker, The Stone-\v{C}ech Compactification, Springer, New York-Berlin, 1974.

\end{thebibliography}
\end{document}